\documentclass[a4paper]{amsart}

\usepackage{amsmath,amssymb,amsthm}
\usepackage{indentfirst,color}
\usepackage[colorlinks,citecolor=red,linkcolor=blue,urlcolor=cyan]{hyperref}

\theoremstyle{plain}
\newtheorem{theorem}{Theorem}[section]
\newtheorem{proposition}[theorem]{Proposition}
\newtheorem{lemma}[theorem]{Lemma}
\newtheorem{corollary}[theorem]{Corollary}
\theoremstyle{definition}
\newtheorem{definition}[theorem]{Definition}
\theoremstyle{remark}
\newtheorem{remark}[theorem]{Remark}
\numberwithin{equation}{section}

\newcommand{\BB}{\mathbb{B}}
\newcommand{\CC}{\mathbb{C}}

\newcommand{\ddbar}{\partial\bar\partial}
\newcommand{\dbar}{\bar{\partial}}
\newcommand{\eps}{\varepsilon}
\newcommand{\calE}{\mathcal{E}}
\newcommand{\calO}{\mathcal{O}}

\newcommand{\Id}{\textup{Id}}

\newcommand{\utr}{\underline{\CC^r}}
\newcommand{\om}{\omega_{\mathrm{e}}}
\DeclareMathOperator{\Dom}{Dom}

\DeclareMathOperator{\supp}{supp}
\newcommand{\inner}[1]{\langle#1\rangle}
\newcommand{\iinner}[1]{\langle\!\langle#1\rangle\!\rangle}

\begin{document}

\title[From $L^2$-Estimates to \(p\)-Positivity]{From the Optimal $L^2$-Estimates \((p,n)\) and \((n,p)\)-Forms to Uniform Griffiths \(p\)-Positivity via Convolution and Restriction}

\author[Z. Liu]{Zhuo Liu}
\address{Mathematical Science Research Center,  Chongqing University of Technology,  No. 69, Hongguang
Avenue, Banan District, Chongqing 400054, China.}
\email{liuzhuo@amss.ac.cn,liuzhuo@cqut.edu.cn}

\author[X.Zhang]{Xujun Zhang}
\address{School of Mathematics and Information Science, Guangxi University, Nanning, Guangxi, 530004, P.R.China}
\address{Guangxi Center for Mathematical Research, Guangxi University, Nanning, Guangxi, 530004, P. R. China}
\address{Tianyuan Mathematical Center in Southwest China, Guangxi Base, Nanning, Guangxi, 530004, P. R. China}
\email{xujunzhang@amss.ac.cn}



\begin{abstract}
We prove that the optimal \(L^2\)-estimate property for \((p,n)\)-forms is preserved under standard convolution regularization. As an application, we show that a singular Hermitian metric satisfying the optimal \(L^2\)-estimate property for \((p,n)\)-forms is uniform Griffiths \(p\)-positive. By establishing the restriction property, we prove that a singular Hermitian metric satisfying the optimal \(L^2\)-estimate property for \((n,p)\)-forms is also uniform Griffiths \(p\)-positive.
\end{abstract}
\maketitle

\section{Introduction}
Plurisubharmonicity and Griffiths/Nakano positivity play fundamental roles in several complex variables and complex geometry, yielding numerous important results, including H\"ormander's $L^2$ estimates for the $\overline{\partial}$-equations (\cite{Dem-82, Hor65}), Ohsawa-Takegoshi's $L^2$-extension theorem (\cite{OT87}),  the optimal $L^2$ extension theorem (\cite{Blocki13, GuanZhou12, GuanZhou15}) and so on. Plurisubharmonic functions, which need not be smooth, offer significant advantages in addressing many problems. Similarly, there has been growing interest in singular metrics with positivity properties for vector bundles.​ A singular Hermitian metric $h$ on a holomorphic vector bundle $E$  is a measurable map from the base manifold $X$ to the space of non-negative Hermitian forms on the fibers  satisfying $0<\det h<+\infty$ almost everywhere. The metric $h$	is called \textit{Griffiths semi-positive}  if $|u|^2_{h^*}$ is plurisubharmonic for any local holomorphic section $u$ of the dual bundle $E^*$ (\cite{BP08, PT18, Rau15}). For line bundles, Griffiths semi-positivity of ​$h$ is equivalent to the plurisubharmonicity of the local weight  $\varphi := -\log h$.

Recently, in \cite{DNW21, DNWZ22, DWZZ18}, Deng, Ning, Wang, Zhang, Zhou established the converse $L^2$ theory by giving alternative characterizations of plurisubharmonicity and Griffiths/Nakano positivity in terms of various $L^2$-conditions for $\overline{\partial}$. These characterizations are significant precisely because they avoid metric differentiations, thereby offering workable definitions of various positivity notions for singular Hermitian metrics on holomorphic vector bundles. For more results on characterizations of positivity, we refer the reader  to \cite{DNW21, DNWZ22, DWZZ18,  DengZhang21, HosonoInayama, Inayama2021, Inayama2022, LMNZ24, LMNZ26, Liu2025, Liu2025v2, Watanabe} and references therein.

The notion of partial positivity also has a rich history. The classical \(p\)-convexity of Andreotti-Grauert \cite{AG62} and Siu's Griffiths \(p\)-positivity \cite{Siu82} are relatively weak ``existence'' conditions: they only require at least \(n-p+1\) positive directions (e.g., \(n-p+1\) positive eigenvalues for line bundles). While this suffices for classical cohomology theories, it is insufficient for the optimal \(L^2\)-estimate property with arbitrary forms, which requires a ``uniform" positivity condition.

For line bundles, this uniform condition precisely means that the sum of  {any} \(p\) eigenvalues is non-negative, which corresponds to the \(p\)-plurisubharmonicity and \(p\)-convexity of Harvey--Lawson \cite{HL09a, HL09b, HL11, HL12, HL13}. Its vector bundle analogues are the Nakano \(p\)-positivity of Siu \cite{Siu82} and the \(p\)-positivity of Demailly \cite{Dem-82}. In parallel, Yang \cite{Yang18, Yang20} introduced and studied the RC-positivity and uniformly RC-positivity, providing another perspective on such existence and uniform positivity conditions.

To state these results precisely, we recall the \(L^2\)-estimate
property. Throughout, \(D\subset\CC^n\) is a domain,
\(\om=i\sum_{j=1}^n dz_j\wedge d\bar z_j\) is the Euclidean K\"ahler
metric, and \(\utr=D\times\CC^r\) is the trivial holomorphic vector
bundle of rank \(r\) over \(D\).

\begin{definition}\label{def:L2est-intro}
We say that \((D,\om;\utr,h)\) satisfies the \emph{optimal
\(L^2\)-estimate property for \((p,n)\)-forms} if for every smooth
strictly \(p\)-plurisubharmonic function \(\phi\) on \(D\), and every
\(\dbar\)-closed \(\utr\)-valued \((p,n)\)-form \(f\) on \(D\) with
finite right-hand side, there exists
\(u\in L^2_{(p,n-1)}(D,\utr,he^{-\phi})\) such that \(\dbar u=f\) and
\begin{equation*} 
  \int_D|u|^2 e^{-\phi}\,dV_{\om}
  \le\int_D\inner{B_{\om,\phi}^{-1}f,f}e^{-\phi}\,dV_{\om},
\end{equation*}
where \(B_{\om,\phi}=[i\ddbar\phi\otimes\Id_{\utr},\Lambda_{\om}]\).
The optimal \(L^2\)-estimate property for \((n,p)\)-forms is defined
analogously.
\end{definition}

For partial positivity in this uniform sense, Inayama \cite{Inayama2022} proved that a smooth function on a domain \(D\subset\CC^n\) is \(p\)-plurisubharmonic with respect to the Euclidean metric \(\om\) if and only if it satisfies the optimal \(L^2\)-estimate property for \((n,p)\)-forms, and also established a characterization of RC-positivity via the optimal \(L^2\)-estimate property for \((n,n)\)-forms. Watanabe \cite{Watanabe} further proved that a smooth Hermitian metric is uniform Nakano \(p\)-semi-positive if and only if it satisfies the optimal \(L^2\)-estimate property for \((n,p)\)-forms, and dual Nakano \(p\)-semi-positive if and only if it satisfies the optimal \(L^2\)-estimate property for \((p,n)\)-forms.

To justify the use of \(L^2\)-estimate conditions as workable substitutes for the curvature positivity of singular Hermitian metrics, it is crucial to verify that they indeed imply geometric positivity. In this direction, Deng--Ning--Wang \cite[Remark~1.9]{DNW21} and Inayama \cite[Question~7.3]{Inayama2022} posed the fundamental question of whether the optimal \(L^2\)-estimate property for \((n,1)\)-forms implies Griffiths positivity for singular Hermitian metrics. This was answered affirmatively by Liu \cite{Liu2025} in the continuous case, and completely resolved by Liu \cite{Liu2025v2} in the general singular case, where it was proved that the optimal or multiple coarse \(L^2\)-estimate property for either \((n,1)\) or \((1,n)\)-forms implies Griffiths positivity.

It is natural to ask whether the analogue of these questions holds for partial positivity. Namely, does a singular Hermitian metric satisfying the optimal \(L^2\)-estimate property for \((p,n)\) or \((n,p)\)-forms necessarily satisfy the corresponding uniform Griffiths \(p\)-positivity? In this paper we answer this question affirmatively under a necessary semi-continuity assumption.

It should be noted that there is a slight difference in the setting compared to \cite{Liu2025v2}. The framework in \cite{Liu2025v2} requires the local \(L^2\)-estimate property to hold on every Stein subdomain. In contrast, in this paper, we consider the trivial bundle over a domain in \(\CC^n\), where only the global solvability of the \(\bar\partial\)-equation is required, and the test weights are merely \(p\)-plurisubharmonic functions. Nevertheless, the method developed in this paper inherits and generalizes the approach introduced in \cite{Liu2025v2}.

Our first main theorem is the following.

\begin{theorem}\label{thm:main-intro}
Let \(h\) be a singular Hermitian metric on
\(\utr\) such that the norm $|s|^2_{h^*}$ is upper semi-continuous for every holomorphic local section $s$ of $\utr^*$. Assume that \((D,\om;\utr,h)\) satisfies the optimal
\(L^2\)-estimate property for \((p,n)\)-forms. Then for every relatively compact open subset \(U\Subset D\), there
exists an increasing sequence \(\{h_k\}\) of smooth dual Nakano \(p\)-semi-positive Hermitian metrics on \(U\) such that $h_k^*$ converging to \(h^*\) in
\(L^1(U)\).
\end{theorem}

Let \(a=(a_{jk})\) be a constant positive definite Hermitian \(n\times n\) matrix, and let
\[
  \omega_a
  = i \sum_{j,k=1}^{n} a_{jk}\, dz_j \wedge d\bar{z}_k .
\]
After a change of coordinates, it becomes the Euclidean metric \(\om\). Apply
Theorem~\ref{thm:dual-nak} and we obtain the following result.

\begin{theorem}\label{cor:anisotropic-intro}
Let \(h\) be a singular Hermitian metric on
\(\utr\) such that the norm $h^*$ is  locally integrable. Assume that \((D,\omega_a;\utr,h)\) satisfies the optimal
\(L^2\)-estimate property for \((p,n)\)-forms. Then for every relatively compact open subset \(U\Subset D\), there
exists an increasing sequence \(\{h_k\}\) of smooth dual Nakano \(p\)-semi-positive Hermitian metrics with respect to \(\omega_a\) on \(U\) such that $h_k^*$ converging to \(h^*\) in
\(L^1(U)\).
\end{theorem}

\begin{remark}
 In fact, \(\{h_k\}\) increasingly convergent to a singular Hermitian metric $\hat h$ which satisfies the optimal
\(L^2\)-estimate property for \((p,n)\)-forms with respect to \(\omega_a\) and equal to $h$ almost everywhere.
\end{remark}

Consequently, we have the following.

\begin{corollary}\label{cor:dual-susc-intro}
Let \(h\) be a singular Hermitian metric on
\(\utr\) such that the norm $|s|^2_{h^*}$ is strongly upper semi-continuous for every holomorphic local section $s$ of $\utr^*$. Assume that $h$ is dual Nakano \(p\)-semi-positive with respect
to \(\omega_a\). Then \(h\) is uniform Griffiths \(p\)-positive with respect to \(\omega_a\).
\end{corollary}

Recall that Griffiths semi-positivity is equivalent to uniform Griffiths \(n\)-semi-positivity with respect to \(\omega_a\) for \textit{all} constant positive definite Hermitian matrices \(a\). Therefore,  we obtain the following result.

\begin{corollary}\label{cor:all-metrics-intro}
Let \(h\) be a singular Hermitian metric on
\(\utr\) such that the norm $|s|^2_{h^*}$ is strongly upper semi-continuous for every holomorphic local section $s$ of $\utr^*$. If  $h$ is dual Nakano
\(p\)-semi-positive with respect to \(\omega_a\) for any constant positive definite Hermitian matrix $a$, then $h$
is Griffiths semi-positive.
\end{corollary}

Moreover, Berndtsson--P\u{a}un \cite[Proposition 3.1]{BP08} showed that a Griffiths semi-negative singular Hermitian metric can be locally approximated from above by a decreasing sequence of smooth Griffiths semi-negative Hermitian metrics via convolution. Then it follows from  H\"ormander's \(L^2\)-existence theorem \cite[Chapter~VIII-(6.1) Theorem]{Demailly2012} and  Corollary \ref{cor:all-metrics-intro} that

\begin{theorem}\label{cor:pseudoconvex-intro}
Let \(D\Subset\CC^n\) be pseudoconvex and \(h\) a singular Hermitian metric on
\(\utr\) such that the norm $|s|^2_{h^*}$ is strongly upper semi-continuous for every holomorphic local section $s$ of $\utr^*$. Then  \(h\) is
Griffiths semi-positive if and only if it satisfies the optimal
\(L^2\)-estimate property for \((n,n)\)-forms with respect to every
K\"ahler metric.
\end{theorem}

In addition, together with Proposition~\ref{prop:restriction-pn},
which states that the optimal \(L^2\)-estimate property for \((n,p)\)-forms
of singular Hermitian metrics is preserved under restriction to
\(p\)-dimensional complex subspaces, Corollary~\ref{cor:dual-susc-intro}
and Proposition~\ref{prop:sing-grif-basic}(c) imply that \(h\) is
uniform Griffiths \(p\)-semi-positive with respect to \(\om\).

\begin{theorem}\label{cor:restriction-intro}
Let \(h\) be a lower semi-continuous Hermitian metric on \(\utr\). Assume that the norm \(|s|^2_{h^*}\) is strongly upper semi-continuous for any local holomorphic section \(s\) of \(\utr^*\), the multiplier submodule sheaf \(\calE(h)\) is coherent, and \((D,\om;\utr,h)\) satisfies the optimal \(L^2\)-estimate property for \((n,p)\)-forms. Then \(h\) is uniform Griffiths \(p\)-semi-positive.
\end{theorem}
 
Finally, by specializing to the classical \((n,1)\)-case and combining with the results of Liu--Zhang \cite[Corollary 1.2]{LZ26}, we generalize Deng--Zhang's theorem \cite[Theorem 1.8]{DengZhang21}, relaxing their \(C^2\)-regularity assumption as follows.

\begin{theorem}\label{thm:generalized-DZ}
Let \(D\subset\CC^n\) be a domain with null thin complement, i.e., \(\overline{D}^\circ = D\), and let \(\varphi\) be a strongly upper semi-continuous function on \(D\). Assume that \((D,\varphi)\) is \(L^2\)-optimal, which means that, for every smooth strictly plurisubharmonic function \(\phi\) and every K\"ahler metric \(\omega\) on \(D\), and for any \(\dbar\)-closed \((n,1)\)-form \(f \in L^2_{(n,1),\mathrm{loc}}(D)\) with a finite right-hand side, there exists a solution \(u\) to \(\dbar u=f\) satisfying the estimate
\begin{equation*}
  \int_D|u|^2_{\omega,\varphi} e^{-\phi} dV_\omega
  \leq
  \int_D \inner{B_\phi^{-1}f,f}_{\omega,\varphi} e^{-\phi} dV_{\omega},
\end{equation*}
where \(B_\phi:=[i\ddbar \phi, \Lambda_{\omega}]\). Then \(D\) is pseudoconvex and \(\varphi\) is plurisubharmonic.
\end{theorem}

The paper is organized as follows. Section~\ref{sec:prelim} reviews \(p\)-plurisubharmonic functions and positivity notions for smooth and singular Hermitian metrics. Section~\ref{sec:conv} proves that convolution preserves the optimal \(L^2\)-estimate property for \((p,n)\)-forms, yielding smooth dual Nakano \(p\)-semi-positive approximations. Section~\ref{sec:restriction} establishes the restriction property of the optimal \(L^2\)-estimate property for \((n,p)\)-forms and deduces uniform Griffiths \(p\)-positivity.

\textbf{Acknowledgement:}
The authors are very grateful to Professor
Xiangyu Zhou for valuable suggestions on related topics.
Zhuo Liu was supported by the National Natural Science Foundation of China (No.12501101) and the Scientific Research Foundation of Chongqing University of Technology (No.2025ZDZ013).
Xujun Zhang was supported by the National Natural Science Foundation of China (No.12501102).

\section{Preliminaries}\label{sec:prelim}
\subsection{\(p\)-plurisubharmonic functions}\label{ss:ppsh}

Let \(D\subset\CC^n\) be a domain and \(\om\) is the Euclidean K\"ahler
metric on \(\CC^n\).

\begin{definition}\label{def:ppsh}
A function \(\phi:D\to[-\infty,+\infty)\) is \emph{\(p\)-plurisubharmonic
with respect to \(\om\)} (or \emph{\(p\)-psh} for short) if it is upper
semi-continuous and its restriction to every \(p\)-dimensional complex
subspace \(S\subset\CC^n\) is subharmonic with respect to \(\om|_S\).
\end{definition}

\begin{remark}\label{rmk:ppsh-basic}
\mbox{}
\begin{enumerate}
\item When \(\phi\) is smooth, \(\phi\) is \(p\)-psh if and only if the
sum of any \(p\) eigenvalues of \(i\ddbar\phi\) with respect to \(\om\)
is nonnegative. Equivalently, \(i\ddbar\phi\wedge\om^{p-1}\) is
semi-positive.

\item A \(p\)-psh function is subharmonic, and hence \textit{strongly upper semi-continuous}, i.e.,
\(
\limsup_{Z\not\ni z\to x} \phi(z)=\phi(x)
\)
for any $x\in D$ and any measure-zero subset $Z\subset D$.   In
addition, if \(\phi\not\equiv-\infty\) is \(p\)-psh and \(\rho_\eps\)
is a standard mollifier on \(\CC^n\), then \(\phi*\rho_\eps\) is
smooth, \(p\)-psh, and decreases to \(\phi\) as \(\eps\to0\).


\item A strongly upper semi-continuous function is \(p\)-psh if and only
if \(i\ddbar\phi\wedge\om^{p-1}\) is semi-positive as a current.

\item If \(\{\phi_j\}\) is a decreasing sequence of \(p\)-psh functions
on \(D\), then \(\phi=\lim_j\phi_j\) is \(p\)-psh.
\end{enumerate}
\end{remark}

\begin{definition}\label{def:ppsh-a}
Let  \(\omega\) be a K\"ahler metric on $D$. A locally
integrable function \(\phi:D\to[-\infty,+\infty)\) is
\emph{\(p\)-psh with respect to \(\omega\)} if it is strongly upper
semi-continuous and \(i\ddbar\phi\wedge\omega^{p-1}\) is semi-positive
as a current.
\end{definition}

\begin{remark}
Verbitsky \cite{V2010} proved that a strongly upper
semi-continuous function \(\phi\) on a K\"ahler manifold is \(p\)-psh with respect to \(\omega\)
if and only if it is subharmonic with respect to \(\omega\) on every
\(p\)-dimensional submanifold.
\end{remark}

\subsection{Curvature positivity for smooth Hermitian metrics}\label{ss:curv}
Let $(E,h)\to D$ be a holomorphic vector bundle and \(\omega\) a K\"ahler
metric on $D$. The Chern curvature is denoted by
\(i\Theta_h\in\Gamma(\Lambda^{1,1}T^*D\otimes\operatorname{End}(E))\).

\begin{definition}\label{def:grif}
Let \(h\) be a smooth Hermitian metric on \(E\). We say that \(h\) is
\emph{uniform Griffiths \(p\)-semi-positive with respect to \(\omega\)} if for
every \(x\in D\), every \(p\)-tuple of \(\omega\)-orthonormal \((1,0)\)-vectors
\(v_1,\dots,v_p\in T_x^{1,0}D\), and every \(s\in E_x\),
\[
  \sum_{j=1}^{p}\langle i\Theta_h(v_j,\bar v_j)s,s\rangle_h\ge 0.
\]
We say that \(h\) is \emph{uniform Griffiths \(p\)-semi-negative  with respect to \(\omega\)} if the reverse
inequality holds.
\end{definition}
\begin{remark}\label{rmk:grif-equiv}
\begin{enumerate}
  \item  When \(p=1\), this is the usual Griffiths semi-positivity. When \(p=n\),
the condition becomes that the trace
\(\operatorname{tr}_\omega(i\Theta_h)\) is a semi-positive Hermitian
form on \(E_x\).
 \item  \(h\) is Griffiths semi-positive \(E\) if
and only if  \(h\) is uniform Griffiths \(n\)-semi-positivity with respect to \(\omega_a\) for all  constant positive definite Hermitian \(n\times n\) matrices \(a\).
\item   \(h\) is uniform Griffiths \(p\)-semi-positive on \(E\) if
and only if the dual metric \(h^*\) is uniform Griffiths \(p\)-semi-negative on
\(E^*\).
  \item  \(h\) is uniform Griffiths \(p\)-semi-positive if and only if for
every local holomorphic section \(s\) of \(E\), the \((p,p)\)-form
\[
  \langle i\Theta_h s,s\rangle_h\wedge\omega^{p-1}
\]
is semi-positive.
 \item  Another notion of uniform positivity is provided by Yang's work. For a line bundle, Yang \cite{Yang18} proved that RC-positivity is equivalent to the existence of a conformal metric \(h\) with \(\operatorname{tr}_\omega(i\Theta_h)> 0\) (i.e., uniform Griffiths \(n\)-positivity). For higher-rank vector bundles, Wang--Yang--Yau \cite{WYY26} established the prescribed Hermitian--Yang--Mills tensor equation: if there exists a metric \(h_0\) with \(\operatorname{tr}_\omega(i\Theta_{h_0})> 0\), one can prescribe the trace of the curvature to be any positive-definite tensor.
\end{enumerate}
\end{remark}

\begin{definition}[\cite{Siu82}]\label{def:nak}
Let \(h\) be a smooth Hermitian metric on \(E\). We say that \(h\) is
\emph{uniform Nakano \(p\)-semi-positive with respect to \(\omega\)} if for
every \(E\)-valued \((n,p)\)-form \(u\), the pointwise pairing
\[
  \langle [i\Theta_h,\Lambda_{\omega}]u,u\rangle_{\omega,h}
\]
is nonnegative. We say that \(h\) is \emph{dual Nakano
\(p\)-semi-positive with respect to \(\omega\)} if for every \(E\)-valued
\((p,n)\)-form \(u\), the pointwise pairing
\[
  \langle [i\Theta_h,\Lambda_{\omega}]u,u\rangle_{\omega,h}
\]
is nonnegative.
\end{definition}

When \(p=1\), the first condition is the usual Nakano semi-positivity and the second is dual Nakano semi-positivity. When \(p=n\), both conditions become that \([i\Theta_h,\Lambda_{\omega}]\) is semi-positive on \((n,n)\)-forms. This is equivalent to \(\operatorname{tr}_\omega(i\Theta_h)\) being a semi-positive Hermitian form, which coincides exactly with uniform Griffiths \(n\)-positivity.

\begin{proposition}\label{prop:nak-grif}
Let \(h\) be a smooth Hermitian metric on \(E\). If \(h\) is
uniform Nakano \(p\)-semi-positive with respect to \(\omega\), then \(h\) is
uniform Griffiths \(p\)-semi-positive with respect to \(\omega\). Similarly, if
\(h\) is dual Nakano \(p\)-semi-positive with respect to \(\omega\), then \(h\) is uniform Griffiths \(p\)-semi-positive with respect to \(\omega\).
\end{proposition}
\begin{proof}
Fix \(x\in X\), a \(p\)-tuple of orthonormal \((1,0)\)-vectors
\(v_1,\dots,v_p\in T_x^{1,0}X\), and \(a\in E_x\). Choose
\(\omega\)-normal coordinates centered at \(x\) with
\(v_j=\partial_{z_j}\) for \(j=1,\dots,p\). Set
\[
  u:=dz_1\wedge\cdots\wedge dz_n\wedge d\bar z_1\wedge\cdots
    \wedge d\bar z_p\otimes a.
\]
This is an \(E\)-valued \((n,p)\)-form. By the curvature formula
\cite[Chapter~VII-(7.1)]{Demailly2012},
\[
  \langle[i\Theta_h,\Lambda_{\omega}]u,u\rangle_h
  =\sum_{j=1}^p\langle\Theta_{j\bar j}a,a\rangle_h
  =\sum_{j=1}^p\langle i\Theta_h(v_j,\bar v_j)a,a\rangle_h.
\]
If \(h\) is uniform Nakano \(p\)-semi-positive, the left-hand side is
nonnegative, hence the sum on the right is nonnegative. Since
\(v_1,\dots,v_p\) and \(a\) are arbitrary, \(h\) is uniform Griffiths \(p\)-semi-positive. The same argument applied to the dual bundle
  with the \((p,n)\)-form
\(d\bar z_1\wedge\cdots\wedge d\bar z_p\wedge dz_1\wedge\cdots\wedge
dz_n\otimes s^*\) gives the dual Nakano \(p\)-semi-positive case.
\end{proof}

\begin{proposition}\label{prop:char-grif}
Let \(h\) be a smooth Hermitian metric on \(E\). Then \(h\) is
uniform Griffiths \(p\)-semi-positive with respect to \(\omega\) if and only if
the norm \(|u|^2_{h^*}\) is \(p\)-psh with respect to \(\omega\) for
every local holomorphic section \(u\) of \(E^*\).
\end{proposition}

\begin{proof}
Fix a local holomorphic section \(u\in\mathcal{O}(E^*)\). By the
Griffiths identity,
\[
  i\ddbar |u|^2_{h^*}
  =\langle\nabla u,\nabla u\rangle_{h^*}
  -\langle\Theta_{h^*}u,u\rangle_{h^*}.
\]
Let \(s\) be the element of \(E\) dual to \(u\) under \(h\). Since
\(\Theta_{h^*}=-\Theta_h^T\), we have
\[
  \langle\Theta_{h^*}u,u\rangle_{h^*}
  =-\langle\Theta_h s,s\rangle_h.
\]
Hence
\[
  i\ddbar |u|^2_{h^*}
  =\langle\nabla u,\nabla u\rangle_{h^*}
  +\langle\Theta_h s,s\rangle_h.
\]
Wedging with \(\omega^{p-1}\), the first term is semi-positive as a
\((p,p)\)-form, so
\[
  i\ddbar |u|^2_{h^*}\wedge\omega^{p-1}\ge0
  \iff
  \langle\Theta_h s,s\rangle_h\wedge\omega^{p-1}\ge0.
\]
We complete the proof.
\end{proof}

\begin{proposition}\label{prop:char-nak}
Let \(h\) be a smooth Hermitian metric on \(E\) over a pseudoconvex
domain \(D\subset\CC^n\). Then:
\begin{enumerate}
\item \(h\) is uniform Nakano \(p\)-semi-positive with respect to \(\omega\)
if and only if \((D,\omega;E,h)\) satisfies the optimal \(L^2\)-estimate
property for \((n,p)\)-forms.
\item \(h\) is dual Nakano \(p\)-semi-positive with respect to
\(\omega\) if and only if \((D,\omega;E,h)\) satisfies the optimal
\(L^2\)-estimate property for \((p,n)\)-forms.
\end{enumerate}
\end{proposition}

\begin{proof}
The forward implications follow from H\"ormander's \(L^2\)-existence
theorem \cite[Chapter~VIII-(6.1) Theorem]{Demailly2012}. The converses
  are due to Watanabe \cite{Watanabe}.
\end{proof}

\subsection{Positivity for singular Hermitian metrics}\label{ss:singular}

\begin{definition}\label{def:sing-grif}
Let \(h\) be a singular Hermitian metric on \(\utr\) such that the norm
\(|s|^2_{h^*}\) is upper semi-continuous for every local holomorphic
section \(s\) of \(\utr^*\). We say that:
\begin{enumerate}
\item[(i)] \(h\) is \emph{uniform Griffiths \(p\)-semi-positive with respect
to \(\om\)} if \(|u|^2_{h^*}\) is \(p\)-psh with respect to \(\om\)
for every local holomorphic section \(u\) of \(\utr^*\).
\item[(ii)] \(h\) is \emph{dual Nakano \(p\)-semi-positive with
respect to \(\om\)} if \((D,\om;\utr,h)\) satisfies the optimal
\(L^2\)-estimate property for \((p,n)\)-forms.
\item[(iii)] \(h\) is \emph{uniform Nakano \(p\)-semi-positive with respect
to \(\om\)} if \((D,\om;\utr,h)\) satisfies the optimal
\(L^2\)-estimate property for \((n,p)\)-forms.
\end{enumerate}
\end{definition}

\begin{proposition}\label{prop:sing-grif-basic}
Let \(h\) be a singular Hermitian metric on \(\utr\), where
\(D\subset\CC^n\) carries the Euclidean metric \(\om\). Assume that the
norm \(|s|^2_{h}\) is upper semi-continuous for every local
holomorphic section \(s\) of \(\utr\).
\begin{enumerate}
\item[(a)] If \(h\) is uniform Griffiths \(p\)-semi-negative, then its
convolution \(h_\delta\) with a standard radical mollifier is also
uniform Griffiths \(p\)-semi-negative.
\item[(b)] If \(\{h_j\}\) is a decreasing sequence of uniform Griffiths
\(p\)-semi-negative singular Hermitian metrics, then the limit \(h=\lim_j h_j\) is
uniform Griffiths \(p\)-semi-negative.

\item[(c)]If \(|s|^2_{h}\) is strongly upper semi-continuous for
every local holomorphic section \(s\) of \(\utr\), and if for almost
every \(p\)-dimensional complex subspace \(S\) the restriction
\(h|_S\) is uniform Griffiths \(p\)-semi-negative, then \(h\) is uniform Griffiths
\(p\)-semi-negative. 

\item[(d)] If \(|s|^2_{h}\) is strongly upper semi-continuous for every local holomorphic section \(s\) of \(\utr\), and there exists a sequence \(\{h_j\}\) of smooth uniform Griffiths \(p\)-semi-negative Hermitian metrics converging to \(h\) in \(L^1_{\mathrm{loc}}\), then \(h\) is uniform Griffiths \(p\)-semi-negative.
\end{enumerate}
\end{proposition}

\begin{proof}
(a) Fix a local holomorphic section \(s\) of \(\utr\). For a standard radical mollifier \(\rho\) on \(\CC^n\), set
\(\rho_\delta(y)=\delta^{-2n}\rho(y/\delta)\), \(y\in\CC^n\)and define
 \[
  |s|^2_{h_\delta}(z)
  :=\int_{\CC^n}|s(z)|^2_{h(z-y)}\rho_\delta(y)\,dV(y).
\]  Since
the translation preserves the Euclidean metric \(\om\), then \(|s|^2_{h_\delta}\)
is \(p\)-psh and hence\(h_\delta\) is uniform Griffiths \(p\)-semi-negative.

(b) For each local holomorphic section \(s\) of \(\utr\), the
functions \(|s|^2_{h_j}\) are \(p\)-psh and decrease to
\(|s|^2_{h}\). A decreasing limit of \(p\)-psh functions is \(p\)-psh,
so \(|s|^2_{h}\) is \(p\)-psh. Hence \(h\) is uniform Griffiths
\(p\)-semi-negative.

(c) Fix a local holomorphic section \(s\) of \(\utr\). We show that
\(|s|^2_{h}\) is \(p\)-psh. Let \(S\) be a \(p\)-dimensional complex
subspace and \(\BB_S(a,r)\Subset D\cap S\) a disk. Since \(|s|^2_{h}\)
is strongly upper semi-continuous, there exist subspaces \(S_j\)
parallel to \(S\), points \(a_j\in S_j\) with
\(\BB_{S_j}(a_j,r)\Subset D\cap S_j\), such that \(a_j\to a\) and
\(|s|^2_{h}(a_j)\to|s|^2_{h^*}(a)\). For almost every \(j\), the
restriction \(|s|^2_{h}|_{S_j}\) is subharmonic by assumption, so
\[
  |s|^2_{h}(a_j)\le
  \frac{1}{\operatorname{Vol}\BB^p(r)}
  \int_{\BB_{S_j}(a_j,r)}|s|^2_{h}.
\]
Letting \(j\to\infty\) and applying Fatou's lemma, together with the
strong upper semi-continuity of \(|s|^2_{h}\),
\[
  |s|^2_{h}(a)
  =\lim_{j\to\infty}|s|^2_{h}(a_j)
  \le\limsup_{j\to\infty}
  \frac{1}{\operatorname{Vol}\BB^p(r)}
  \int_{\BB_{S_j}(a_j,r)}|s|^2_{h}
  \le\frac{1}{\operatorname{Vol}\BB^p(r)}
  \int_{\BB_S(a,r)}|s|^2_{h}.
\]
Thus \(|s|^2_{h}|_S\) satisfies the mean value inequality and hence subharmonic. Since \(S\) is arbitrary,
\(|s|^2_{h}\) is \(p\)-psh. Since \(s\) is arbitrary, \(h\) is
uniform Griffiths \(p\)-semi-negative.

(d) Fix a local holomorphic section \(s\) of \(\utr\). We need to show that \(|s|^2_{h}\) is \(p\)-psh. Since \(h_j \to h\) in \(L^1_{\mathrm{loc}}\), we have \(|s|^2_{h_j} \to |s|^2_{h}\) in \(L^1_{\mathrm{loc}}\). By passing to a subsequence if necessary, we may assume that \(|s|^2_{h_j} \to |s|^2_{h}\) almost everywhere. Let \(Z \subset D\) be the measure-zero set where the convergence fails.

Let \(S \subset \CC^n\) be an arbitrary \(p\)-dimensional complex subspace, and let \(\BB_S(a,r) \Subset D \cap S\) be a \(p\)-dimensional ball. By (c), we may assume \(Z\cap S\) is a measure-zero subset of \(S\cap D\). Since \(|s|^2_{h_j}\) are smooth \(p\)-psh, their restrictions to \(S\) satisfy the mean value inequality:
\[
  |s|^2_{h_j}(a) \le \frac{1}{\operatorname{Vol}\BB^p(r)} \int_{\BB_{S}(a,r)} |s|^2_{h_j}.
\]
Letting \(j \to \infty\), for any $a\in (S\cap D)\setminus Z$, by Fatou's lemma, we obtain that \[
  |s|^2_{h}(a) \le \frac{1}{\operatorname{Vol}\BB^p(r)} \int_{\BB_{S}(a,r)} |s|^2_{h}.
\] Since \(|s|^2_{h}\) is strongly upper semi-continuous, we can approximate any \(a \in S\cap D\) by a sequence \(a_k \in (S\cap D)\setminus Z\) with \(a_k \to a\) and \(|s|^2_h(a_k) \to |s|^2_h(a)\). Passing to the limit, \(|s|^2_h|_S\) satisfies the mean value inequality for all \(a \in S\cap D\). Thus \(|s|^2_h|_S\) is subharmonic, which implies that \(|s|^2_h\) is \(p\)-psh.
\end{proof}

\section{From \((p,n)\)-estimates to uniform Griffiths \(p\)-positivity via convolution}\label{sec:conv}

Let \(h\) be a singular Hermitian metric on \(\utr\) such that
\(h^*=h^{-1}\) is locally integrable. For a standard radical mollifier \(\rho\) on \(\CC^n\), set
\(\rho_\delta(y)=\delta^{-2n}\rho(y/\delta)\) and define, for
\(U\Subset D\) and \(\delta>0\) small,
\begin{equation}\label{eq:hdelta}
  h_\delta^{-1}(z)=\int_{\CC^n}\rho_\delta(y)h^{-1}(z-y)\,dV(y),
  \qquad
  h_\delta(z)=(h_\delta^{-1}(z))^{-1}.
\end{equation}

For \(\phi\) smooth, set
\[
  B_\phi=[i\ddbar\phi\otimes\Id_{\utr},\Lambda_{\om}]
  \quad\text{on }\Lambda^{p,n}T^*D\otimes\utr.
\]
We also let \(\dbar^*_{he^{-\phi}}\) denote the formal adjoint of
\(\dbar\) with respect to
\(\iinner{\alpha,\beta}_{he^{-\phi}}=\int_D\inner{\alpha,\beta}_{\om, h}
e^{-\phi}\,dV_{\om}\).

The following two statements are the \((p,n)\)-analogues of
\cite[Lemma 3.1 and Lemma 3.2]{Liu2025v2}. The proofs are identical to
the \((1,n)\)-case treated there, and we omit them.

\begin{lemma}\label{lem:Bochner}
Assume that \((D,\om;\utr,h)\) satisfies the optimal \(L^2\)-estimate
property for \((p,n)\)-forms. Then for every
\(\alpha\in\Dom(\dbar^*_{he^{-\phi}})\cap L^2_{(p,n)}(D,\utr,he^{-\phi})\),
\begin{equation*} 
  \iinner{B_\phi\alpha,\alpha}_{he^{-\phi}}
  \le\|\dbar^*_{he^{-\phi}}\alpha\|_{he^{-\phi}}^2.
\end{equation*}
\end{lemma}

\begin{lemma}\label{lem:adj}
For \(\beta\in C_c^\infty(D,\Lambda^{p,n}\otimes\utr)\), set
\(\eta=e^{-\phi}h_\delta^T\beta\), \(\alpha=e^{\phi}h^{-T}\eta\). Then
\begin{equation*} 
  \dbar^*_{he^{-\phi}}\alpha=e^{\phi}h^{-T}\dbar_0^*\eta,\qquad
  \dbar^*_{h_\delta e^{-\phi}}\beta=e^{\phi}h_\delta^{-T}\dbar_0^*\eta.
\end{equation*}
\end{lemma}

\begin{theorem}\label{thm:conv}
Assume that \((D,\om;\utr,h)\) satisfies the optimal \(L^2\)-estimate
property for \((p,n)\)-forms. Let \(\phi\) be a smooth strictly
\(p\)-psh function on an open neighbourhood of
\(\overline D\). Then for every \(\delta>0\) sufficiently small and
every \(\beta\in C_c^\infty(D,\Lambda^{p,n}\otimes\utr)\) with
\[
  \supp\beta\Subset D_\delta:=\{z\in D:\operatorname{dist}(z,\partial D)>\delta\},
\]
the Bochner inequality
\begin{equation*} 
  \iinner{B_\phi\beta,\beta}_{h_\delta e^{-\phi}}
  \le\|\dbar^*_{h_\delta e^{-\phi}}\beta\|_{h_\delta e^{-\phi}}^2
\end{equation*}
holds.
\end{theorem}

\begin{proof}
Let \(\beta\in C_c^\infty(D,\Lambda^{p,n}\otimes\utr)\) with
\(\supp\beta\Subset D_\delta\).
For \(|y|<\delta\), set \(D^y:=D+y\), \(h^y(z):=h(z-y)\), so that
\(h^y\) is a Hermitian metric on \(D^y\). Since translation preserves
\(\om\) and \((D,\om;\utr,h)\) satisfies the optimal \(L^2\)-estimate
property, the pair \((D^y,\om;\utr,h^y)\) satisfies the same property.
Set
\[
  \eta:=e^{-\phi}h_\delta^T\beta, \qquad
  \alpha:=e^{\phi}(h^y)^{-T}\eta.
\]
Applying Lemma~\ref{lem:Bochner} to \((D^y,\om;\utr,h^y)\) with the
weight \(\phi\) (which is \(p\)-psh on a neighbourhood of
\(\overline D\), hence on \(D^y\) for small \(y\)) and the datum
\(\alpha\) gives
\begin{equation}\label{eq:conv-base}
  \iinner{B_\phi e^{\phi}(h^y)^{-T}\eta,\eta}_{0,D^y}
  \le\iinner{e^{\phi}(h^y)^{-T}\dbar_0^*\eta,\dbar_0^*\eta}_{0,D^y}.
\end{equation}
Since \(\eta\) and \(\dbar_0^*\eta\) are both supported in
\(\supp\beta\), and \(\supp\beta\subset D\cap D^y\) for
\(|y|<\delta\), both integrands in \eqref{eq:conv-base} vanish outside
\(D\cap D^y\), so the integrals over \(D^y\) equal the integrals
over \(D\):
\[
  \iinner{B_\phi e^{\phi}(h^y)^{-T}\eta,\eta}_{0,D}
  \le\iinner{e^{\phi}(h^y)^{-T}\dbar_0^*\eta,\dbar_0^*\eta}_{0,D}.
\]
Multiply by \(\rho_\delta(y)\) and integrate over \(y\in\CC^n\). By
Fubini's theorem and \eqref{eq:hdelta}, we obtain
\[
  \iinner{B_\phi e^{\phi}h_\delta^{-T}\eta,\eta}_{0,D}
  \le\iinner{e^{\phi}h_\delta^{-T}\dbar_0^*\eta,\dbar_0^*\eta}_{0,D}.
\]
Then by Lemma \ref{lem:adj}, we complete the proof.
\end{proof}

We observe that in the proofs of Deng--Ning--Wang--Zhou \cite{DNWZ22} (characterizing smooth Nakano semi-positive metrics via the optimal \(L^2\)-estimate for \((n,1)\)-forms) and Watanabe \cite{Watanabe} (characterizing dual Nakano \(p\)-semi-positive metrics via \((p,n)\)-forms), the test weights are smooth strictly plurisubharmonic functions on a neighborhood of \(\overline{D}\), and the test forms are compactly supported smooth \((p,n)\)-forms. Theorem~\ref{thm:conv} provides precisely the Bochner inequality for such test forms supported in \(D_\delta\). Therefore, applying their arguments directly yields the following theorem.

\begin{theorem}\label{thm:dual-nak}
Let \(h\) be a singular Hermitian metric on \(\utr\) such that
\(h^*\) is  locally integrable. Assume
that \((D,\om;\utr,h)\) satisfies the optimal \(L^2\)-estimate property
for \((p,n)\)-forms. Then for every sufficiently small \(\delta>0\), the
smooth Hermitian metric \(h_\delta\) is dual Nakano \(p\)-semi-positive on \(D_\delta\).
\end{theorem}

By Theorem~\ref{thm:dual-nak}, the dual metric \(h_\delta^*\) is smooth and uniform Griffiths \(p\)-semi-negative on \(D_\delta\), and it converges to \(h^*\) in \(L^1_{\mathrm{loc}}\). If we assume that \(h^*\) is strongly upper semi-continuous, then applying Proposition~\ref{prop:sing-grif-basic}(d) to \(h_\delta^*\) yields the following corollary.

\begin{corollary}\label{cor:dual-to-grif-p}
Let \(h\) be a singular Hermitian metric on \(\utr\) such that the norm \(|s|^2_{h^*}\) is strongly upper semi-continuous for every local holomorphic section \(s\) of \(\utr^*\). Assume that \((D,\om;\utr,h)\) satisfies the optimal \(L^2\)-estimate property for \((p,n)\)-forms. Then \(h\) is uniform Griffiths \(p\)-semi-positive.
\end{corollary}

\section{From \((n,p)\)-estimates to uniform Griffiths \(p\)-positivity via restriction}\label{sec:restriction}

In this section, we aim to show that the optimal \(L^2\)-estimate property for \((n,p)\)-forms implies uniform Griffiths \(p\)-positivity. It should be noted that, as demonstrated by \cite[Example 6.3]{LXYZ24}, the standard convolution technique does not preserve Nakano positivity. Consequently, we cannot directly apply the convolution technique on the ambient space to achieve the result. Instead, we must reduce to the restriction of the bundle to \(p\)-dimensional complex subspaces, where the \((n,p)\)-estimate descends to a \((p,p)\)-estimate. The core strategy is to first establish the preservation of the \(L^2\)-estimate property under restriction for lower semi-continuous Hermitian metrics with coherent multiplier submodule sheaves.

First, we recall the definition of the multiplier submodule sheaf and the Lebesgue-type differentiation theorem from \cite{Liu2025v2}.

\begin{definition}[\cite{Cataldo98}]
    The multiplier submodule sheaf $\mathcal{E}(h)$ of $\mathcal{O}(E)$ associated to $(E,h)$  is the sheaf of the germs $s_x \in \mathcal{O}(E)_x$
    such that $|s_x|^2_h$ is integrable in some neighborhood of $x$.
\end{definition}

\begin{remark}
  Assume that \(h^*\) is locally bounded and \((D,\om;\utr,h)\) satisfies the optimal \(L^2\)-estimate property for \((n,1)\)-forms. Following the standard arguments of \cite{Dem-93JDG, Nadel}, one can obtain the coherence of the multiplier ideal sheaf \(\mathcal{E}(h)\) (see \cite{HosonoInayama, Inayama2022}).
\end{remark}

\begin{lemma}[{\cite[Lemma 2.13]{Liu2025v2}}]\label{lem:weak-siu}
Let \(h\) be a Hermitian metric on \(\utr\) of rank \(r\) over a
domain \(U\times D'\subset\CC^{n-k}_{w'}\times\CC^k_{w''}\). Assume that
\(\calE(h)=\calO(\utr)\). Then for almost every \(z''\in D'\), for every
relatively compact subset \(V\Subset U\), every continuous section
\(f\in\Gamma(\overline V\times D',\Omega^{n,p}\otimes\utr)\), and every continuous
positive Hermitian matrix function
\(P\in C^0(\overline V\times D',\mathrm{Herm}^+\utr)\), we have
\[
  \int_V \inner{Pf,f}_{h(w',z''),\om'} \,dV_{\om'}(w')<+\infty
\]
and
\begin{align*}
   &\lim_{\eps\to0}\frac{1}{\mathrm{Vol}(\BB_{z''}(\eps))}\int_{V\times{\BB_{z''}(\eps)}} \inner{Pf,f}_{h(w',z''),\om}(w',w'')   \, dV_{\om}(w)
  \\= &    \int_V|\inner{Pf,f}_{h(w',z''),\om'}|\,dV_{\om'}(w').
\end{align*}
\end{lemma}

\begin{proposition}\label{prop:restriction-pn}
Let \(h\) be a lower semi-continuous Hermitian metric on \(\utr\). Assume that $\calE(h)=\calO(\utr)$  and  \((D,\om;\utr,h)\)
satisfies the optimal \(L^2\)-estimate property for \((n,p)\)-forms. Fix a \(p\)-dimensional complex subspace \(S=\{w''=0\}\subset\CC^n\),
with coordinates \(w=(w',w'')\in\CC^p\times\CC^{n-p}\). Then for almost
every \(z''\in S^\perp\) with \(S_{z''}\cap D\neq\emptyset\),  and every \(p\)-dimensional ball
\(B\Subset S_{z''}\cap D\), \((B,\om|_B;\utr,h)\) satisfies the optimal \(L^2\)-estimate property for \((p,p)\)-forms.
\end{proposition}

\begin{proof}
 Fix \(z''\in A\), where \(A\subset S^\perp\) is the full-measure set of
Lemma~\ref{lem:weak-siu}.   Denote $\om'=\om|_B$. Fix a \(p\)-dimensional ball
\(B\Subset S_{z''}\cap D\) and  a smooth strictly p-psh function \(\phi\) on  \( B\) and a \(\dbar\)-closed
\(f\in L^2_{(p,p)}(B;he^{-\phi})\) with
\[
   \int_B\inner{B_{\phi}^{-1}f,f}_{\om',h}
    e^{-\phi}\,dV_{\om'}<+\infty.
\]

Since every \((p,p)\)-form on \(B\) is
automatically \(\dbar\)-closed. By  by
truncation, convolution and a standard approximation argument, we may assume that
\(f \in C_c^\infty\cap L^2_{(p,p)}(B;he^{-\phi})\).
After translating and rotating coordinates, we may assume
\(B=\BB^p_{0'}(r)\), \(\supp f\Subset\BB^p_{0'}(r_0)\) with \(r_0<r\) and $\BB^p_{0'}(r)\times \BB^{n-p}_{z''}(\delta_0)\Subset D$.
Choose \(r_0<r_1<r_2<r\), $\delta<\delta_0/2$ and cut-offs
\(\chi_1\in C_c^\infty(\BB^p_{0'}(r_2))\) with \(\chi_1\equiv1\) on
\(\BB^p_{0'}(r_1)\), and \(\chi_2\in C_c^\infty(\BB^{n-p}_{z''}(2\delta))\)
with \(\chi_2\equiv1\) on \(\BB^{n-p}_{z''}(\delta)\). Set
\(\chi:=\chi_1\chi_2\), \(dW''=dw''_1\wedge\cdots\wedge dw''_{n-p}\),
and choose \(\alpha\in C^\infty(B,\wedge^{p,p-1}\otimes\utr)\) with
\(\dbar_{w'}\alpha=f\). Define
\[
  F:=\dbar(\chi\,dW''\wedge\alpha)\in
  C_c^\infty(D,\wedge^{n,p}\otimes\utr).
\]
Since \(\dbar\alpha=f\) and \(\dbar f=0\), we have \(\dbar F=0\), and
\(F=dW''\wedge f\) on \(\BB^p_{0'}(r_1)\times\BB^{n-p}_{z''}(\delta)\).

In addition, by the Richberg approximation theorem, there is a smooth strictly p-psh function \(\tilde\phi\) on $D$ such that $\tilde\phi (w',w'')=\phi(w')+C|w''-z''|^2$ on $\BB^p_{0'}(r_0)\times \BB^{n-p}_{z''}(\delta_0)$.

Now we choose smooth convex nondecreasing \(\tau_1,\tau_2:[0,\infty)\to
[0,\infty)\) with \(\tau_1=0\) on \([0,r_1^2]\), \(\tau_1>0\) on
\((r_1^2,\infty)\), and \(\tau_2=0\) on \([0,\delta^2]\),
\(\tau_2>0\) on \([\delta^2,\infty)\). Set
\[
  \Phi(w',w''):=\tau_1(|w'|^2)+\tau_2(|w''-z''|^2),
\]
then \(\Phi\) is smooth plurisubharmonic.

By the optimal \(L^2\)-estimate for $(n,p)$-form on \(D\) with weight \(k\Phi+\tilde\phi \), there
exists \(v_{k,\delta}\in L^2_{(n,p-1)}(D,\utr;he^{-\Phi_k};\mathrm{loc})\) with
\(\dbar v_{k,\delta}=F\) and
\[
  \int_D|v_{k,\delta}|^2_{\om,h}e^{-k\Phi-\tilde\phi }\,dV_{\om}
  \le  \int_D\inner{B_{k\Phi+\tilde\phi }^{-1}F,F}_{\om,h}
    e^{-k\Phi-\tilde\phi }\,dV_{\om}.
\]
Write \(v_{k,\delta}=v'_{k,\delta}\wedge dW''+v''_{k,\delta}\), where \(v''_{k,\delta}\) contains at least
one \(d\bar w''\). Since \(\om\) is Euclidean, we have
\(|v_{k,\delta}|^2_{\om,h}\ge|v'_{k,\delta}|^2_{\om,h}\) and  \(\dbar_{w'}v'_{k,\delta}=f\)
 on \(\BB^p_{0'}(r_1)\times\BB^{n-p}_{z''}(\delta)\).

 Noticing that \(\Phi= 0\) on \(\BB^p_{0'}(r_1)\times\BB^{n-p}_{z''}(\delta)\) and  \(\Phi> 0\) on the support of $\dbar\chi$, we obtain that
\begin{align*}
    & \lim_{k\to\infty}\int_D\inner{B_{k\Phi+\tilde\phi }^{-1}F,F}_{\om,h}
    e^{-k\Phi-\tilde\phi }\,dV_{\om} \\
  = & \int_{\BB^p_{0'}(r_1)\times\BB^{n-p}_{z''}(\delta)}\inner{B_{\tilde\phi }^{-1}f\wedge dW'',f\wedge dW''}_{\om',h}e^{-\tilde\phi }\,dV_{\om'}.
\end{align*}
Letting $k\to \infty$, we can obtain $\{v'_\delta\}$ such that  \(\dbar_{w'}v'_{\delta}=f\) and
\[
  \int_D|v'_{\delta}|^2_{\om,h}e^{-\tilde\phi }\,dV_{\om}
  \le  \int_{\BB^p_{0'}(r_1)\times\BB^{n-p}_{z''}(\delta)}\inner{B_{\tilde\phi }^{-1}f\wedge dW'',f\wedge dW''}_{\om',h}e^{-\tilde\phi }\,dV_{\om'}.
\]
Moreover, it follows from Fubini's theorem that there exists
\(\xi_\delta\in\BB^{n-p}_{z''}(\delta)\) such that the slice
\(v'_{\delta}(\cdot,\xi_\delta)\) satisfies
\(\dbar_{w'}v'_k(\cdot,\xi_\delta)=f\) on \(B\) and
\begin{align*}
    & \int_B|v'_\delta(w',\xi_\delta)|^2_{\om',h}e^{-\phi}\,dV_{\om'} \\
\le & \frac{1}{\mathrm{Vol}(\BB^{n-p}_{z''}(\delta))}\int_{\BB^p_{0'}(r_1)\times\BB^{n-p}_{z''}(\delta)}\inner{B_{\tilde\phi }^{-1}f\wedge dW'',f\wedge dW''}_{\om,h}e^{-\tilde\phi }\,dV_{\om}.
\end{align*}

Noticing that the \((n,p)\)-form \(f\wedge dW''\) has antiholomorphic indices only in
$w'$,  we have $$\inner{B_{\tilde\phi }^{-1}f\wedge dW'',f\wedge dW''}_{\om,h}=\inner{B_{\phi }^{-1}f ,f}_{\om',h}.$$

Since \(\tilde\phi(w',z'')=\phi(w')\) on \(\BB^p_{0'}(r_1)\) and \(h\) has a positive lower bound locally, by the weak compactness of \(L^2_{\mathrm{loc}}\), we can choose a sequence \(\delta_m \to 0\) such that \(v'_{\delta_m}(w',\xi_{\delta_m})\) converges weakly to a limit \(u \in L^2_{(p,p-1)}(B, \om'; \utr, h)\). Then \(\dbar u=f\).

In addition, since \(h\) is lower semi-continuous, for any compact subset \(K\Subset B\), passing to a finite subcover of \(K\), we may assume there is a sequence of continuous Hermitian metrics \(h_j\) on \(K\) increasing to \(h\). Then by Fatou's lemma, we have
\begin{align*}
   &\int_{K} |u|^2_{\om',h_j}e^{-\phi}dV_{\om'}  \\  = & \int_{K}\lim_{\delta_m\to0}  |v'_{\delta_m}|^2_{\om',h_j(w',\xi_{\delta_m}) }e^{-\phi}dV_{\om'}   \\
   \le  &\liminf_{\delta_m\to0} \int_{K}  |v'_{\delta_m}|^2_{\om',h_j(w',\xi_{\delta_m}) }e^{-\phi} dV_{\om'}  \\
   \le & \liminf_{\delta_m\to0} \int_{K}  |v'_{\delta_m}|^2_{\om',h(w',\xi_{\delta_m}) }e^{-\phi} dV_{\om'}  \\
   \le & \liminf_{\delta_m\to0}  \frac{1}{\mathrm{Vol}(\BB^{n-p}_{z''}(\delta_m))}\int_{\BB^p_{0'}(r_1)\times\BB^{n-p}_{z''}(\delta_m)}\inner{B_{\tilde\phi }^{-1}f\wedge dW'',f\wedge dW''}_{\om,h}e^{-\tilde\phi }\,dV_{\om} \\
   = &  \int_B\inner{B_{\phi }^{-1}f, f}_{\om',h}e^{-\phi}dV_{\om'},
\end{align*}
where the last equality follows from Lemma~\ref{lem:weak-siu} as \(\delta_m \to 0\). Since \(h_j\) increases to \(h\) and \(K\Subset B\) is arbitrary, we obtain
\[
  \int_B|u|^2_{\om',h}e^{-\phi}\,dV_{\om'}
  \le\int_B\inner{B_{\phi}^{-1}f,f}_{\om',h}e^{-\phi}\,dV_{\om'}.
\]
\end{proof}

By combining the restriction property established above with the coherence hypothesis on the multiplier submodule sheaf, we can now rigorously pass from the optimal \(L^2\)-estimate property for \((n,p)\)-forms to uniform Griffiths \(p\)-positivity.

\begin{theorem}\label{thm:main}
Let \(h\) be a lower semi-continuous Hermitian metric on \(\utr\). Assume that the norm \(|s|^2_{h^*}\) is strongly upper semi-continuous for any local holomorphic section \(s\) of \(\utr^*\), the multiplier submodule sheaf \(\calE(h)\) is coherent, and \((D,\om;\utr,h)\) satisfies the optimal \(L^2\)-estimate property for \((n,p)\)-forms. Then \(h\) is uniform Griffiths \(p\)-semi-positive.
\end{theorem}

\begin{proof}
Since \(\calE(h)\) is coherent, its quotient sheaf \(\calO(\utr)/\calE(h)\) is supported on a proper analytic subset \(Z \subset D\). Outside \(Z\), we have \(\calE(h)=\calO(\utr)\).
By \cite[Chapter VIII-(7.3)]{Demailly2012}, analytic subsets are \(L^2\)-negligible. Thus, the optimal \(L^2\)-estimate property for \((n,p)\)-forms is preserved on \(D\setminus Z\) with respect to smooth strictly \(p\)-psh weights on \(D\).
This allows us to apply Proposition \ref{prop:restriction-pn} to the domain \(D\setminus Z\), which implies that for almost every parallel \(p\)-dimensional subspace \(S_{z''}\), the restriction \(h|_{S_{z''}\setminus Z}\) satisfies the optimal \(L^2\)-estimate property for \((p,p)\)-forms. By Corollary \ref{cor:dual-to-grif-p}, this is equivalent to \(h|_{S_{z''}\setminus Z}\) being uniform Griffiths \(p\)-semi-positive.
Finally, since \(|s|^2_{h^*}\) is strongly upper semi-continuous, Proposition \ref{prop:sing-grif-basic}(c) and (d) allows us to cross the analytic set \(Z\) and upgrade the slice-wise positivity to the entire domain \(D\). Hence, \(h\) is uniform Griffiths \(p\)-semi-positive.
\end{proof}

\begin{remark}
It is unknown  whether the solvability of the \(\bar{\partial}\)-equation for \((n,p)\)-forms with optimal \(L^2\)-estimates alone implies the coherence of the multiplier submodule sheaf. A related question, posed by Inayama \cite[Conjecture 1.1]{Inayama2022}, asks whether Griffiths semi-positivity implies the coherence of multiplier submodule sheaves.   Recently, Inayama \cite{Inayama22} proved that when the unbounded locus of \(\det h\) is discrete, the Griffiths semi-positivity of \((E,h)\) implies the coherence of \(\calE(h)\). Zou \cite{Zou22} proved that this coherence holds when \(\det h\) has analytic singularities. Furthermore, Yang \cite{Yang Nagoya} proved that the coherence is guaranteed when the non-integral points of \((\det h)^2\) are discrete.
\end{remark}


\begin{thebibliography}{99}

\bibitem{AG62}
A. Andreotti, H. Grauert:
Th\'eor\`emes de finitude pour la cohomologie des espaces complexes.
Bull. Soc. Math. France {\bf 90}: 193--259 (1962).
 
\bibitem{BP08} B. Berndtsson, M. P\u{a}un: Bergman kernels and the pseudoeffectivity of relative canonical bundles. Duke Math. J. {\bf 145}: 341--378 (2008).

\bibitem{Blocki13} Z. B{\l}ocki: Suita conjecture and the Ohsawa-Takegoshi extension theorem. Invent. Math. {\bf 193}(1): 149--158 (2013).

\bibitem{Cataldo98}		 M. A. de~Cataldo:		Singular Hermitian metrics on vector bundles.
		{J. Reine Angew. Math.} {\bf 502}: 93--122 (1998).

\bibitem{Dem-82}
 J.~P.~Demailly: Estimations {$L^2$} pour l'op\'erateur $\bar\partial$ d'un fibr\'e
  vectoriel holomorphe semi-positif au-dessus d'une vari\'et\'e k\"ahl\'erienne
  compl\`ete. Ann. Sci. \'{E}cole Norm. Sup.   {\bf 15}(3): 457--511 (1982).

\bibitem{Dem-93JDG}  J.~P.~Demailly,
	  A numerical criterion for very ample line bundles.
	  J. Differential Geom. {\bf 37}(2):323--374 (1993).

\bibitem{Demailly2012}
J. P. Demailly: Complex analytic and differential geometry, (2012). Available at
 \url{https://www-fourier.ujf-grenoble.fr/~demailly/manuscripts/agbook.pdf}.

\bibitem{DNW21} F. Deng, J. Ning, Z. Wang: Characterizations of plurisubharmonic functions. Sci. China. Math. {\bf 64}(9): 1959--1970 (2021).

\bibitem{DNWZ22} F. Deng, J. Ning, Z. Wang, X. Zhou: Positivity of holomorphic vector bundles in terms of $L^p$-estimates for $\overline{\partial}$. Math. Ann. {\bf 385}: 575--607 (2023).

\bibitem{DNWZ23}	F. Deng, J. Ning, Z. Wang, X. Zhou:
		Approximation and extension of Hermitian metrics on holomorphic vector bundles over Stein manifolds.  C. R. Math. Acad. Sci. Paris {\bf 362}: 1707--1715 (2024).

\bibitem{DWZZ18} F. Deng, Z. Wang, L. Zhang, X. Zhou: New characterizations of plurisubharmonic functions and positivity of direct image sheaves. Amer. J. Math. {\bf 146}(3): 751--768 (2024).

\bibitem{DengZhang21} F. Deng, X. Zhang: Characterizations of curvature positivity of Riemannian vector bundles and convexity or pseudoconvexity of bounded domains in $\mathbb{R}^n$ or $\mathbb{C}^n$ in terms of $L^2$-estimate of $d$ or $\overline{\partial}$ equation. J. Funct. Anal. 281(9): no. 109184 (2021).

\bibitem{GuanZhou12} Q. Guan, X. Zhou: Optimal constant problem in the $L^2$ extension theorem. C. R. Math. Acad. Sci. Paris {\bf  350}(15-16): 753--756 (2012).

\bibitem{GuanZhou15} Q. Guan, X. Zhou: A solution of an $L^2$ extension problem with an optimal estimate and applications. Ann. of Math.(2) {\bf 181}(3): 1139--1208 (2015).
    
 \bibitem{HL09a}
F. R. Harvey, H. B. Lawson, Jr.:
An introduction to potential theory in calibrated geometry.
Amer. J. Math. {\bf 131}(4): 893--944 (2009).

\bibitem{HL09b}
F. R. Harvey, H. B. Lawson, Jr.:
Duality of positive currents and plurisubharmonic functions in calibrated geometry.
Amer. J. Math. {\bf 131}(5): 1211--1239 (2009).   

\bibitem{HL11}
F. R. Harvey, H. B. Lawson, Jr.:
Plurisubharmonicity in a general geometric context.
In:  {Geometry and analysis. No. 1}, volume 17 of  {Adv. Lect. Math. (ALM)}, pages 363--402. Int. Press, Somerville, MA, 2011.

\bibitem{HL12}
F. R. Harvey, H. B. Lawson, Jr.:
Geometric plurisubharmonicity and convexity: an introduction.
 {Adv. Math.}, 230(4-6): 2428--2456, 2012.

\bibitem{HL13}
F. R. Harvey, H. B. Lawson, Jr.:
\(p\)-convexity, \(p\)-plurisubharmonicity and the Levi problem.
 {Indiana Univ. Math. J.}, 62(1): 149--169, 2013.


\bibitem{Hor65} L. H\"ormander: $L^2$ estimates and existence theorems for the $\overline{\partial}$ operator. Acta Math. {\bf 113}: 89--152 (1965).

\bibitem{HosonoInayama} G. Hosono, T. Inayama: A converse of H\"ormander's $L^2$-estimate and new positivity notions for vector bundles. Sci. China Math. 64(8): 1745-1756 (2021).

\bibitem{Inayama2021} T. Inayama: From H\"ormander's $L^2$-estimates to partial positivity. C. R. Math. Acad. Sci. Paris {\bf 359}: 169--179 (2021).


\bibitem{Inayama2022} T. Inayama: Nakano positivity of singular Hermitian metrics and vanishing theorems of Demailly-Nadel-Nakano type. Algebr. Geom. {\bf 9}(1): 69--92 (2022).


\bibitem{Inayama22}  T. Inayama: Singular Hermitian metrics with isolated singularities. Nagoya Math. J. {\bf 248}: 980--989 (2022).


\bibitem{LMNZ24}  Z. Li, X. Meng, J. Ning, X. Zhou:  A converse to the Skoda $L^2$ division theorem.
 \url{arXiv:2405.18713} (2024).

\bibitem{LMNZ26}  Z. Li, X. Meng, J. Ning, X. Zhou:  A converse to the Skoda $L^2$ division theorem II.
 \url{arXiv:2609.33300} (2026).

 \bibitem{Liu2025} Z. Liu:An optimal  \(L^2\) extension for continuous \(L^2\)-optimal Hermitian metrics. \url{arXiv:2506.22840} (2025).

\bibitem{Liu2025v2} Z. Liu: Convolution regularization preserves the \(L^2\)-estimate property for \((1,n)\)-forms. \url{arXiv:2506.22834v2} (2026).

 
\bibitem{LXYZ24} Z. Liu, B. Xiao, H. Yang, X. Zhou: Multiplier Submodule Sheaves and a problem of Lempert.  \url{arXiv:2111.13452v2} (2024).

\bibitem{LXYZ24} Z. Liu, B. Xiao, H. Yang, X. Zhou: Multiplier submodule sheaves, $L^2$ extension and a problem of Lempert. J. Reine Angew. Math. {\bf 833}:  111--157 (2026).
    
 \bibitem{LZ26} Z. Liu, X. Zhang:  Skoda's $L^2$ division theorem for $L^2$-optimal pairs. To appear in Anal. PDE.  \url{arXiv:2411.15670} (2024).
     

\bibitem{Nadel}
{A. M. Nadel},  Multiplier ideal sheaves and K\"ahler-Einstein metrics of positive scalar curvature.
  Proc. Nat. Acad. Sci. U.S.A. {\bf 86} :7299--7300 (1989), and  Ann. of Math. (2) {\bf 132}(3):549--596 (1990).

\bibitem{OT87} T. Ohsawa, K. Takegoshi: On the extension of $L^2$ holomorphic functions. Math. Z. {\bf 195}(2): 197--204
(1987).

\bibitem{PT18} M. P\u{a}un, S. Takayama: Positivity of twisted relative pluricanonical bundles and their direct images. J. Algebraic Geom. {\bf 27}: 211--272 (2018).

\bibitem{Rau15} H. Raufi: Singular hermitian metrics on holomorphic vector bundles. Ark. Mat. {\bf 53}(2): 359--382 (2015).

\bibitem{Siu82}
Y.-T. Siu:
Complex-analyticity of harmonic maps, vanishing and Lefschetz theorems.
J. Diff. Geom. {\bf 17}: 55--138 (1982).

\bibitem{V2010} M. Verbitsky: Plurisubharmonic functions in calibrated geometry and $q$-convexity. 
Math. Z. {\bf 264}(4): 939--957 (2010).

\bibitem{WYY26}
M. Wang, X. Yang, S.-T. Yau:
Existence of Hermitian metrics with prescribed Hermitian-Yang-Mills tensors I.
\url{arXiv:2603.10611} (2026).

\bibitem{Watanabe} Y. Watanabe: Curvature operator of holomorphic vector bundles and $L^2$-estimate condition for $(n,q)$ and $(p,n)$-forms. Tohoku Math. J. (2) {\bf 77}(4): 445--471 (2025).

\bibitem{Yang Nagoya} H. Yang: Coherence of multiplier submodule sheaves on generic complex tori. Nagoya Math. J. {\bf 260}: 666--673 (2025).

\bibitem{Yang18}
X. Yang:
RC-positivity, rational connectedness and Yau's conjecture.
Camb. J. Math. {\bf 6}(2): 183--212 (2018).

\bibitem{Yang20}
X. Yang:
RC-positive metrics on rationally connected manifolds.
Forum Math. Sigma {\bf 8}: Paper No. e53, 19 pp (2020).


\bibitem{Zou22} Y.  Zou: Note on the singular Hermitian metrics with analytic singularities. To appear in  Ann. Inst. Fourier (Grenoble). \url{arXiv:2209.05053} (2022).




\end{thebibliography}
\end{document}